\documentclass[11pt]{amsart}
\usepackage{mathtools, amssymb, amsfonts, amsthm, enumitem, pdfsync}
\usepackage{epsfig, psfrag, color, multicol, graphicx}
\usepackage{amsfonts}
\usepackage{epstopdf}

\newtheorem{thm}{Theorem}[section]
\newtheorem{lemma}[thm]{Lemma}

\newtheorem{cor}[thm]{Corollary}
\newtheorem{prop}[thm]{Proposition}

\theoremstyle{plain}
\newtheorem{theo}[thm]{Theorem}
\newtheorem{lem}[thm]{Lemma}
\newtheorem{conj}[thm]{Conjecture}

\theoremstyle{definition}

\numberwithin{equation}{section}

\def\sq{\square}

\def\zz{\mathbb Z}
\def\nn{\mathbb N}

\def\rr{\mathbb R}
\def\qqq{\mathbb Q}

\def\De{\Delta}

\def\la{\lambda}
\def\ga{\gamma}

\def\de{\delta}
\def\ep{\ve}
\def\al{\alpha}
\def\be{\beta}
\def\om{\omega}

\def\ve{\varepsilon}

\def\ssu{\subset}

\def\<{\langle}
\def\>{\rangle}

\def\Z{ {\text {\rm Z} } }

\def\Q{{\text {\rm Q} } }

\def\0{{\mathbf 0}}

\def\.{\hskip.06cm}
\def\ts{\hskip.03cm}

\def\vol{{\text {\rm vol}}}
\def\conv{{\text {\rm {conv}} }}

\def\bx{{\textbf{x}}}

\def\by{{\textbf{y}}}

\def\poly{\textup{\textsf{P}}}
\def\FP{{\textup{\textsf{FP}}}}

\def\Z{\mathbb{Z}}
\def\R{\mathbb{R}}
\def\N{\mathbb{N}}
\def\Q{\mathbb{Q}}

\newcommand{\cj}[1]{\overline{#1}}

\def\a{\cj{a}}

\newcommand{\x}{\mathbf{x}}

\newcommand{\z}{\mathbf{z}}

\newcommand{\floor}[1]{\lfloor#1\rfloor}

\newcommand{\ex}{\exists\ts}

\def\nin{\noindent}

\newcommand{\sfrac}[2]{\small{\text{$\frac{#1}{#2}$}}}
\def\same{\equiv}

\def\NP{{\textup{\textsf{NP}}}}
\def\PSPACE{{\textup{\textsf{PSPACE}}}}

\def\sharpP{\textup{\textsf{\#P}}}

\def\KK{\om}
\def\LL{\sigma}

\newcommand{\problem}[1]{\textsc{#1}}

\newcommand{\problemdef}[3]{
\bigskip
\begin{tabular}{p{0.1\textwidth} p{0.8\textwidth}}
\multicolumn{2}{l}{\problem{#1}}\\
\textbf{Input:} & {#2} \\
\textbf{Decide:} & {#3}
\end{tabular}
\bigskip
}

\newcommand{\countingdef}[3]{
\bigskip
\begin{tabular}{p{0.1\textwidth} p{0.8\textwidth}}
\multicolumn{2}{l}{\problem{#1}}\\
\textbf{Input:} & #2 \\
\textbf{Output:} & #3
\end{tabular}
\bigskip
}

\renewcommand{\mod}[1]{
\;\, (\textup{mod} \; #1)
}

\def\cpl{\ts\backslash\ts}

\def\mmod{\;\text{mod}\;}

\usepackage{url}

\title[On the number of integer points in convex polyhedra]
{On the number of integer points in \\ translated and expanded polyhedra}

\author[Danny Nguyen \and Igor Pak]{Danny Nguyen$^{\star}$ \and Igor~Pak$^{\star}$}

\thanks{\thinspace ${\hspace{-.45ex}}^\star$Department of Mathematics,
UCLA, Los Angeles, CA, 90095.
\hskip.06cm
Email:
\hskip.06cm
\texttt{\{ldnguyen,\ts{pak}\}@math.ucla.edu}}

\thanks{\today}

\begin{document}
\maketitle

\begin{abstract}
We prove that the problem of minimizing the number of integer points in
parallel translations of a rational convex polytope in $\rr^6$ is \NP-hard.
We apply this result to show that given a rational convex polytope $P\ssu \rr^6$,
finding the largest integer $t$ s.t.\ the expansion $tP$ contains fewer than $k$ integer points is also \NP-hard.
We conclude that the Ehrhart quasi-polynomials of rational polytopes
can have arbitrary fluctuations.
\end{abstract}

\vskip.7cm

\section{Introduction}

In integer and combinatorial optimization, many problems are
computationally hard when the dimension is unbounded. In fixed
dimensions, the situation is markedly different as many classical
problems become tractable. Notably Lenstra's algorithm for
\emph{Integer Programming}, and Barvinok's algorithm for
\emph{counting integer points} in finite dimensional rational polytopes
are polynomial.

In recent years, there has been a lot of work, including by the
authors, to show that many problems in bounded dimension remain
computationally hard as soon as one leaves the classical framework
(see below). This paper proves hardness of two integer optimizations
problems related to translation and expansion of rational polytopes
in bounded dimensions.

We then consider the problem of describing \emph{Ehrhart quasi-polynomials}
of rational polytopes.  These quasi-polynomials are of fundamental
importance in both discrete geometry and integer optimization,
yet they remain somewhat mysterious and difficult to study.
We apply our result to prove a rather surprising property: that
Ehrhart quasi-polynomials of rational polytopes can have arbitrary
\emph{fluctuations} of consecutive values (see below).

\smallskip

\subsection{Translation of polytopes}
Let $\nn=\{0,1,2,\ldots\}$.
The following problem was considered
by Eisenbrand and H\"{a}hnle in~\cite{EH}.

\problemdef{Integer Point Minimization (IPM)}
{$A \in \qqq^{m\times n}$, a rational polyhedron $Q \subset \rr^m$, $k \in \nn$.}
{$\exists \ts b \in Q$ \, s.t. \. $\#\{\bx\in \zz^n\ts{}:\ts{}A \bx\le b\} \ts \le \ts k$?}

\nin
\emph{Parametric polytopes} \ts $P_b:= \{x\in \rr^n\ts{}:\ts Ax\le b\}$ \ts
were introduced by Kannan~\cite{K1}, who gave a polynomial time
algorithm for \textsc{IPM} with $k=0$ and~$n$ bounded.
For larger fixed values $k$, Aliev, De Loera and Louveaux \cite{ADL} proved that \textsc{IPM} is also polynomial time by employing the
\emph{short generating functions} technique by Barvinok and
Woods~\cite{BW} (see also~\cite{B2,B3}).
The following problem is an especially attractive special case:

\problemdef{Polytope Translation}
{$A \in \qqq^{m\times n}$, \, $b \in \qqq^m$, \, $\vec v \in \qqq^n$,  and \ts $k \in \nn$.}
{$\exists \ts \la, \, 0\le \la \le 1$ \, s.t. \ts $\#\{\bx\in \zz^n\ts{}:\ts{}A(\bx- \la\vec v)\le b\} \ts \le \ts k$?}

\nin
In terms of parametric polytopes, this asks for a translation of the original polytope $P$ by $\la\vec v$ so that it has at most $k$ integer points.  \textsc{Polytope Translation}
is a special case of the \textsc{Integer Point Minimization} problem,
when $Q$ is $1$-dimensional.

\smallskip

Eisenbrand and H\"{a}hnle proved that the \textsc{Polytope Translation}
is $\NP$-hard for $n=2$ and $m$ unbounded:

\begin{theo}[\cite{EH}]\label{th:EH}
Given a rational $m$-gon \ts $Q \subset \R^{2}$,  minimizing \ts $|Q + \lambda \vec e_{1}|$ over $\lambda \in \R$ is
$\NP$-hard.
\end{theo}

Here and everywhere below, $|P|$ denotes the number of integer points
in a polytope~$P$, and $\vec e_{1} = (1,0,\dots)$ is the standard first coordinate vector.
We prove a similar result for $n=6$ with a \emph{fixed}
number~$m$ of vertices.

\begin{theo}\label{th:real_min}
Given a rational polytope $P \subset \R^{6}$ with at most $64$ vertices,
minimizing $|P + \lambda \vec e_{1}|$ over $\lambda \in \R$ is $\NP$-hard.
\end{theo}

This resolves a problem by Eisenbrand.\footnote{F.~Eisenbrand, personal
communication (September 2017).}
Since the dimension is fixed, the number of facets of $P$ is at most an explicit constant.
An integer version of this is:

\begin{theo}\label{th:int_min}
Given a rational polytope $P \subset \R^{6}$ with at most $60$ vertices
and an integer $N \in \N$,
minimizing \ts $|P + t \vec e_{1}/N|$ \ts over $t \in \Z$ is $\NP$-hard.
\end{theo}

While Theorem~\ref{th:int_min} is implied by Theorem~\ref{th:real_min} by a simple argument on rationality, its proof is simpler and will be presented first (cf.\ Section~\ref{s:real}).
The technique differs from those in~\cite{EH} and our earlier
work on the subject.

To prove Theorem~\ref{th:int_min}, we show how to embed
a classical $\NP$-hard quadratic optimization problem into \textsc{Polytope Translation}.
This is done by viewing each term in the quadratic objective as the
integer volume of a separate polygon in $\rr^2$, which are then merged
in a higher dimension into a single convex polytope (cf.~\cite{KannanNPC,shortPR}).
Let us mention that positivity and convexity
are major obstacles here, and occupy much of the proof.

\medskip

\subsection{Expansions of polytopes}
A \emph{quasi-polynomial} $p(t):\Z\to \Z$ is an integer function
$$
p(t) \. = \. c_0(t) \ts  t^d \ts + \ts  c_1(t) \ts  t^{d-1} \ts + \ts \ldots \ts + \ts  c_d(t),
$$
where $c_i(t)$, $0\le i \le d$, are periodic with integer period, and $c_{d}(t) \not\same 0$.
We call $d$ the degree of the quasi-polynomial.
For a rational polytope $P \subset \R^{n}$ of full dimension,
consider the counting function:
$$
f_{P}(t) \; \coloneqq \; \bigl|tP \cap \Z^{n}\bigr|.
$$
Ehrhart famously proved that $f_{P}(t)$ is a degree $d$ quasi-polynomial, called the
\emph{Ehrhart quasi-polynomial}, see e.g.~\cite[$\S$18]{B3}.
Furthermore, it is also known (and not hard to see) that $c_{0}(t) = \vol_{n}(P)$, or equivalently $f_{P}(t)\sim \vol_n(P) \ts t^{n}$.

Many interesting combinatorial problems can be restated in
the language of Ehrhart quasi-polynomials.
We start with the following classical problem:

\countingdef{Frobenius Coin Problem}
{$\cj\al = (\al_1,\ldots, \al_n) \in \nn^{n}$, \
  $\gcd(\al_{1},\dots,\al_{n})=1$.}
{$g(\a) \coloneqq \max \big\{ t \in \N \; : \; $ $\nexists \.
c_1,\ldots, c_n \in \nn$ \, s.t. \ts $t=c_1 \al_1+\ldots + c_n \al_n \big\}$.}

\nin
In other words, this  problem asks for the largest integer~$t$ that cannot
be written as a combination of the coins $\al_{i}$'s.  Such a $t$ exists by
the $\gcd(\cdot)=1$ condition.
Finding $g(\a)$ is an $\NP$-hard problem when the dimension $n$ is not bounded, see~\cite{RA-NP}.
For a fixed~$n$, Kannan proved that the problem can be solved in polynomial
time~\cite{K2,BW}.

\medskip

We can restate the \textsc{Frobenius Coin Problem} as follows.
Let
$$
\Delta_{\cj \al} \; \coloneqq \; \{\x \in \R^{n} \; : \; \cj \al \cdot \x = 1,\;
\x \ge 0\} \quad  \text{and}  \quad f_{\cj \al} \. \coloneqq \. f_{\Delta_{\cj \al}}\,.
$$
Then $f_{\cj \al}(t)$ counts the number of ways to write $t \ge 0$
as an $\N$-combination of the $\al_{i}$'s.  Thus, $g(\cj \al)$ is
the largest $t \ge 0$, such that \ts $f_{\cj \al}(t) = 0$.
Beck and Robins~\cite{BR} used this setting to consider the
following generalization:

\countingdef{$k$-Frobenius Problem}
{$\cj\al = (\al_1,\ldots, \al_n) \in \nn^{n}$,\
  $\gcd(\al_{1},\dots,\al_{n})=1$,\
$k \in \nn$.}
{$g(\al,k) \coloneqq \max \big\{ t \in \N \; : \; f_{\cj \al}(t) < \ts k \big\}$.}

\nin
In other words, the problem asks for the largest integer $t$ that cannot be represented
as a combinations of~$\al_i$'s in $k$ different ways.  Aliev, De Loera and Louveaux~\cite{ADL}
generalized Kannan's theorem to prove that for fixed $n$ and $k$ the
problem is still in~$\poly$.  Motivated by the above interpretation with the
simplex $\Delta_{\cj \al}$, they also considered the following generalization:

\countingdef{$k$-Ehrhart Threshold Problem ($k$-ETP)}
{A rational polytope $P \subset \R^{n}$ and  $k \in \nn$.}
{$g(P,k) \coloneqq \max \big\{ t \in \N \; : \; f_{P}(t) < \ts k \big\}$.}

\nin
For a polytope $P$, this asks for the largest $t$ so that $tP$ contains
fewer than $k$ integer points.  Again, when both $n$ and $k$ are fixed,
it was shown in~\cite{ADL} that this problem is in~$\poly$. However,
for varying $k$ we have:

\begin{theo}\label{t:Ehrhart}
The \textsc{$k$-ETP} is $\NP$-hard for rational polytopes $P \subset \R^{6}$
with at most $60$ vertices.
\end{theo}

It is an open problem whether the
\textsc{$k$-Frobenius Problem} is $\NP$-hard when $k$ is a
part of the input (see~$\S$\ref{ss:finrem-min}).

\medskip

\subsection{Fluctuations of the Ehrhart quasi-polynomial}  It is well
known that every quasi-polynomial \ts $p(t):\zz\to \zz$ \ts can be
written in the form:
\begin{equation}\label{eq:quasi}
p(t) \; = \; \sum_{i=1}^{r} \. \gamma_{i} \ts \prod_{j=1}^{n} \ts\bigl\lfloor\al_{ij}t + \be_{ij}\bigr\rfloor\ts,
\end{equation}
where \ts $\al_{i},\be_{i},\ga_{i} \in \Q$.
The smallest $n$ for which $p(t)$ is representable in this form is called the \emph{degree}
of~$f(t)$.
It is also known how to compute $f_{P}(t)$ in the from~\eqref{eq:quasi} efficiently
when~$n$ is fixed (see e.g.~\cite{VW}).

Not all $n$ quasi-polynomial arise as Ehrhart quasi-polynomials of full-dimensional polytopes $P \subset \R^{n}$.
For instance, $p(t) = 1 + t\floor{\frac{t}{2}} - t\floor{\frac{t-1}{2}}$ cannot
be an Ehrhart quasi-polynomial because $p(t) > 0$ for all $t$,
yet its leading term fluctuates between odd and even values of~$t$.
However, when restricted to finite intervals, every quasi-polynomial
can be realized as $f_{P}$ of a polytope $P$, in the following sense:

\begin{theo}\label{th:fluctuation} 
Let $N \in \N$ and $p : \Z \to \Z$ be a quasi-polynomial of the form~\eqref{eq:quasi},
with $\gamma_{i} \in \Z$, $\al_{ij},\be_{ij} \in \Q$ for $1\le i \le r$ 
and $1\le j \le n$.
Then there exists a rational polytope \ts $Q \subset \R^{d}$ \ts and integers \ts
$K,M \in \N$, such that:
$$
p(t) + K \. = \. f_{Q}(t + M) \quad \text{for every} \quad 0 \le t < N.
$$
Moreover, we have \ts $d = O(n + \lceil \log r \rceil)$, and polytope $Q$ 
has at most $r 4^{n+1}$ vertices. Here the vertices of $Q$ and the 
constants $K,M$ can be computed in polynomial time.
\end{theo}

Roughly, this theorems say that locally, Ehrhart quasi-polynomials
can fluctuate as badly as general quasi-polynomials.
In particular, we have:

\begin{cor}\label{c:univ}
For every sequence $c_{0},\dots,c_{r-1} \in \N$, there exists a polytope 
$Q \subset \R^{d}$ and $K,M \in \N$ such that:
$$
c_{i} + K \; = \; f_{Q}(i + M) \quad\text{for every}\quad 0 \le i < r.
$$
Moreover, we have \ts  $d = O(\log r)$ and polytope $Q$ has at most $O(r)$ vertices.  
Here the vertices of $Q$ 
and the constants $K,M$ can be computed in polynomial time.
\end{cor}

\begin{proof}
  Consider the degree $1$ quasi-polynomial
$$f(t) \, = \, \sum_{i=0}^{r-1} \. c_{i} \left(\left\lfloor\frac{t-i}{r}\right\rfloor - \left\lfloor\frac{t-i-1}{r}\right\rfloor \right).
$$
  Then $f(i) = c_{i}$ for $0 \le i < r$.
  Now we apply Theorem~\ref{th:fluctuation} to $f(t)$ with $N = r$.
\end{proof}

\medskip

\subsection{Brief historical overview}
\textsc{Integer Programming} (\textsc{IP}) asks
for given $A\in \qqq^{m\times n}$ and $b \in \qqq^m$, to decide whether
$$
\exists \ts \x\in \zz^n \ \. : \ \. A\ts\x \le b\ts.
$$
Equivalently, the problem ask whether a rational polytope contains
an integer point.

When $n$ is unbounded, this problem includes \textsc{Knapsack} as a
special case, and thus \NP-complete (see e.g.~\cite{GJ}).
For fixed~$n$, the situation is drastically different.
Lenstra~\cite{L} famously showed that \textsc{IP} is in~\poly,
even when $m$ is unbounded (see also~\cite{Schrijver}).
Barvinok \cite{B1} showed that the corresponding counting
problem is in~\FP, pioneering a new technique in this setting
(see also~\cite{B3,B4}).

The \textsc{Parametric Integer Programming} (\textsc{PIP}) asks
for a given $A\in \qqq^{m\times n}$, $B\in \qqq^{m\times \ell}$
and $b \in \qqq^m$, to decide whether
$$
\qquad
\forall \ts\by \in Q\cap \zz^\ell \quad \exists \ts\bx\in \zz^n \ \, : \ \, A\ts\bx \. + \. B\ts\by \, \le \,  b\ts,
$$
where $Q\ssu \qqq^\ell$ is a convex polyhedron given by $\ts K\ts\by \le u$, for some
$K \in \qqq^{\ell\times r}$, $u \in \qqq^r$.   Kannan showed that \textsc{PIP}
is in~\poly~(see also~\cite{ES,shortPR}).  In~\cite{K2}, Kannan used the \textsc{PIP}
interpretation to show that for a fixed number~$\ell$ of coins, the
\textsc{Frobenius Coin Problem} is in~\poly.  Barvinok and Woods~\cite{BW}
showed that the corresponding counting problem is in~\FP, but only when the
dimensions~$\ell$ and~$n$ are fixed (see also~\cite{W1}).


Although the above list is not exhaustive, most other problems in this area
with fixed dimensions are computationally hard, especially in view of our
recent works.
Let us single out one negative small-dimensional result.
We showed in~\cite{KannanNPC} that given two
rational polytopes \ts $P, Q \subset \R^3$, it is $\sharpP$-complete
to compute
$$\#\ts\bigl\{x \in \Z \  : \ \ex \z \in \Z^{2}\,, \quad (x,\z) \in P\setminus Q \bigr\}\ts.
$$
Note that the corresponding decision problem is a special case of \textsc{PIP},
and thus can be decided in polynomial time.  This elucidated the limitations of the
Barvinok--Woods approach (see also~\cite{shortGF,shortPR}).

\smallskip

The Frobenius problem and its many variations is thoroughly discussed in~\cite{RA},
along with its connections to lattice theory, number theory and convex polyhedra.
There are also some efficient practical algorithms for solving it, see~\cite{BHNW}.
The \textsc{$k$-Frobenius Problem}, also called the \emph{generalized Frobenius problem},
has been intensely studied in recent years, see e.g.~\cite{AHL,FS}.

\smallskip

Ehrhart quasi-polynomials become polynomials for integer polytopes, in which case there
is a large literature on their structure and properties (see e.g.~\cite{B3,B4}
and references therein).  We discuss integer polytopes in Section~\ref{s:int-pol}. 
A bounded number of leading coefficients of Ehrhart
quasi-polynomials in arbitrary dimensions can be computed in polynomial time~\cite{B-quasi}
(see also~\cite{B+}).  There is also some interesting analysis of the
periods of the coefficients $c_i(t)$, see~\cite{BSW,W2}.
It seems that fluctuations of Ehrhart quasi-polynomials have not been considered until now.

\medskip

\subsection{Notations}
As mentioned earlier, $|P|$ always denote the number of integer points
in a convex polytope $P\ssu\rr^n$.  We use $P+\vec w$ to denote translation
of $P$ by vector~$\vec w$.
The first coordinate vector $(1,0,\dots)$ is denoted by $\vec e_{1}$.

When the ambient space $\R^{n}$ is clear, we use $\{x_{i} = \xi_{i},\dots,x_{j} = \xi_{j}\}$
to denote the subspace with specified coordinates $x_{i}=\xi_{i},\dots,x_{j}=\xi_{j}$.
We write $f(t)\gg g(t)$ for $g(t)=o\bigl(f(t)\bigr)$ as $t\to \infty$.  Finally, we use the notations $\nn=\{0,1,2,\ldots\}$ and $\zz_+=\{1,2,\ldots\}$.

\bigskip

\section{Proof of Theorem~\ref{th:int_min}}\label{sec:int_min}

\subsection{General setup}
We start with the following classical problem:

\problemdef{Quadratic Diophantine Equations (QDE)}
{$\al,\be,\gamma \in \N$.}
{$\exists \ts u \in \N, \, 0\le u < \ga$ \, s.t. \ts $u^{2} \equiv \al \mod \be$?}

\nin
Manders and Adleman~\cite{MA} proved that QDE is \NP-complete
(see also~\cite[$\S$7.2]{GJ}).  Observe that the problem remains
\NP-complete when we assume $\al,\gamma < \be$,
Thus, the problem can be rephrased as the problem of minimizing
\begin{equation}\label{eq:QDE_restated}
f(u,v) \. \coloneqq \. (u^{2} - \al - \be v)^{2} \quad \text{over} \quad
(u,v) \in \textup{B} \cap \Z^{2}.
\end{equation}
where $\textup{B} = [0,\ga) \times [0,\be)$.
Indeed, we have $\min_{(u,v) \in \textup{B}} f(u,v) = 0$ if and only if the
congruence in \textsc{QDE} is feasible.

Let $N = \be\ga$.  The two variables $(u,v) \in \textup{B}$ can be encode
into a single integer variable $0 \le t < N$ by:
$$
u = \floor{t/\be} \quad\text{and}\quad v = t \mod{\be} = t - \be \floor{t/\be}.
$$
It is clear that each pair $(u,v) \in \textup{B} \cap \Z^{2}$ corresponds to such a unique $t \in [0,N-1]$ and vice versa.
So we can restate the problem as minimizing
$f \big( \floor{t/\be},  t - \be\floor{t/\be} \big)$ over $t \in [0,N)$.
Now we have:
\begin{align}
f \big( \floor{t/\be},  t - \be\floor{t/\be} \big) \;
=& \;\; \Big( \ts \floor{t/\be}^{2} - \al - \be \big(t - \be\floor{t/\be} \big) \ts \Big)^{2} \nonumber\\
=& \;\; \Big( \ts \floor{t/\be}\big(\floor{t/\be} + \be^{2}\big) - \big( \al + \be t \big) \ts \Big)^{2} \nonumber \\
=& \;\; \underbrace{\floor{t/\be}^{2} \big( \be^{2} + \floor{t/\be} \big)^{2}}_{T_{1}(t)}  \,+\,  \underbrace{\big( \al + \be t \big)^{2}}_{T_{2}(t)}  \,-\,  \underbrace{2 \floor{t/\be} \big(\be^{2} + \floor{t/\be} \big) \big(\al + \be t \big)}_{S(t)} \label{eq:T12}
\end{align}
Here we denote by $T_{1}(t),T_{2}(t)$ and $S(t)$ the three terms in the above sum.
First, we need to convert $-S(t)$ into a positive term.
Fix a large constant~$\LL$, say \ts $\LL := 10\be^{5}$ \ts will suffice for our purposes.
We have:
\begin{align}
- S(t) \;
=& \;\; -S(t) \,+\, 2 \be(\be^{2}+\be)(\al + \be t) \,-\, 2 \be(\be^{2}+\be)(\al + \be t) \,+\, \LL \,-\, \LL \nonumber\\
=& \;\; \Bigl[ \be(\be^{2}+\be) - \floor{t/\be} \bigl(\be^{2} + \floor{t/\be} \bigr) \Bigr] 2 (\al + \be t) \,+\, \LL \,-\, 2 \be(\be^{2}+\be)(\al + \be t)  - \LL \nonumber \\
=& \;\; \underbrace{\big( \be - \floor{t/\be} \big) \big( \be^{2} + \be + \floor{t/\be} \big)  (2\al + 2\be t)}_{T_{3}(t)} \,+\, \underbrace{\bigl[ \LL \,-\, 2 \be(\be^{2}+\be)(\al + \be t) \bigr]}_{T_{4}(t)}  \,-\, \LL. \label{eq:T34}
\end{align}
Thus,
$$f \big( \floor{t/\be},  t - \be\floor{t/\be} \big) \; = \; T_{1}(t) \,+\, T_{2}(t) \,+\, T_{3}(t) \,+\, T_{4}(t) \,-\, \LL.$$
Note that $T_{1}(t),\dots,T_{4}(t) > 0$ for $0 \le t < N$.
Let
$$g(t) \. := \. \LL \ts + \ts f \bigl( \floor{t/\be},  t - \be\floor{t/\be} \bigr).
$$
We can rephrase the original $\NP$-hard problem as the problem of computing the following minimum:
\begin{equation}\label{eq:gT}
\min_{0 \le t < N} g(t) \; = \; \min_{0 \le t < N} T_{1}(t) \,+\, \dots \,+\, T_{4}(t).
\end{equation}
Note that each function $T_{i}(t)$ is a product of terms of the form $p \pm q t$ or $r \pm \floor{t/\be}$ for some constants $p,q,r > 0$.
We encode each of these three types of functions as the number of integer points in some translated polytope.
From this point on, we assume that $0 \le t < N$, unless stated otherwise.


\subsection{Trapezoid constructions}\label{sec:trapezoid}
To illustrate the idea, we start with the simplest function $qt$ with  $q \in \Z_{+}$.
Let $\ep = 1/4N^{2}$ and $\vec v = \vec e_{1}/N = (1/N,0,\dots,0)$.
 Consider the following triangle:
$$
  \Delta \; = \; \big\{(x,y) \in \R^{2} \;:\; x,y \ge \ep,\,  qN(1-x) \ge y\big\}.
$$
(see Figure~\ref{f:Delta}).
Fix a line \ts $\ell := \{x=1\}$.
It is easy to see that the hypotenuse of $\Delta + t \vec v$ intersects $\ell$ at the point $y = qN t/N = qt$.
So we have $(\Delta + t \vec v) \cap \ell = [\ep, qt]$, and thus $|\Delta + t \vec v| = qt$.

\medskip

To encode a function $p + q t$ with  $p,q \in \Z_{+}$, we take $\Delta$ and extend vertically by a distance $p-\frac{1}{2}$ below the line $y=0$ to make a trapezoid $F_{A}$.  Similarly, to encode a function $p' - q t$ with $p' > qN$, we translate the hypotenuse of $\Delta$ up by $2\ep$, and then extend upward by $p'$ to get a trapezoid $F_{B}$ (see Figure~\ref{f:Delta}).  Formally, let:
$$\aligned
F_{A} \; & = \; \big\{(x,y) \in \R^{2} \;:\; 1 \,\ge\, x \,\ge\, \ep,\,  qN(1-x) \,\ge\, y \,\ge\, 1/2-p\big\} \ \ \text{and} \\
F_{B} \; & = \; \big\{(x,y) \in \R^{2} \;:\; 1 \,\ge\, x \,\ge\, \ep,\,  p' \,\ge\, y \,\ge\, qN(1-x) + 2\ep\}.
\endaligned
$$

\begin{figure}[hbt]
\vspace{-1em}
\begin{center}
\psfrag{ep}{$\ep$}
\psfrag{2ep}{$2\ep$}
\psfrag{O}{\scriptsize $O$}
\psfrag{1}{\scriptsize $1$}
\psfrag{slope}{\scriptsize slope = $qN$}
\psfrag{Trig}{$\Delta$}
\psfrag{ell}{$\ell$}
\psfrag{1/N}{\scriptsize$\frac{1}{N}$}
\psfrag{2/N}{\scriptsize$\frac{2}{N}$}
\psfrag{...}{$\dots$}

\psfrag{FA}{$F_{A}$}
\psfrag{FB}{$F_{B}$}
\psfrag{height1}{\scriptsize $p-\frac{1}{2}$}
\psfrag{height2}{\scriptsize $p'$}

\epsfig{file=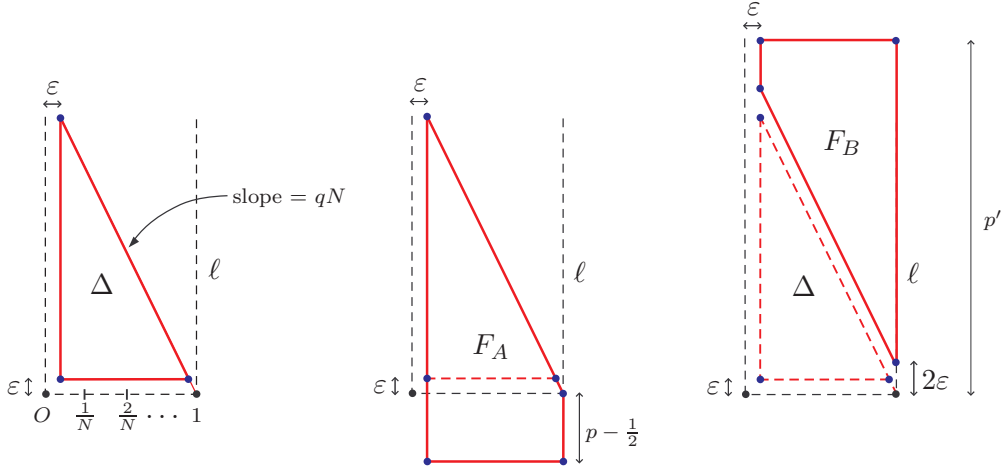, height=7cm}
\end{center}
\vspace{-1em}
\caption{The triangle $\Delta$ and trapezoids $F_{A}, F_{B}$. }
\label{f:Delta}
\vspace{-1.5em}
\end{figure}

\nin
Let us show that these trapezoids encode the function as stated above.
For $F_{A}$, we have $(F_{A} + t \vec v) \cap \ell = \big[\frac{1}{2}-p,\, q t\big]$, and thus $|F_{A} + t \vec v| = p + q t$.
For $F_{B}$, the hypotenuse of $F_{B} + t \vec v$ intersects $\ell$ at $qt + 2\ep$.
So we have $(F_{B} + t \vec v) \cap \ell = \big[qt + 2\ep,\, p'\big]$, and thus $|F_{B} + t \vec v| = p' - q t$ as desired.

\medskip


For the function $\floor{t/\be}$, we can encode it with the following triangle:
$$
\Delta' \; = \; \big\{(x,y) \in \R^{2} \;:\; x,y \ge \ep,\,  \ga(1-x) \ge y\big\}.
$$
(see Figure~\ref{f:Deltaprime}).
It is easy to see that the hypotenuse of $\Delta' + t \vec v$ intersects $\ell$ at the point $y=\gamma t / N = t/\be$.
So $(\Delta'+t\vec v) \cap \ell = [\ep, t/\be]$ and thus $|\Delta'+t \vec v| = \floor{t/\be}$.

\medskip

By modifying $\Delta'$ and keeping the same slope $\ga$, we can encode the functions
\ts $r + \floor{t/\be}$ \ts and \ts $r' - \floor{t/\be}$ \ts with \ts
$r,r' \in \Z_{+}$, $r' > \ga$, by using the following trapezoids:
$$\aligned
F_{C} \; & = \; \big\{(x,y) \in \R^{2} \;:\; 1 \,\ge\, x \,\ge\, \ep,\;  \ga(1-x) \,\ge\, y \,\ge\, 1/2-r\big\} \ \ \text{and}\\
F_{D} \; & = \; \big\{(x,y) \in \R^{2} \;:\; 1 \,\ge\, x \,\ge\, \ep,\;  r' \,\ge\, y \,\ge\, \ga(1-x) + 2\ep \big\},
\endaligned
$$
respectively (see Figure~\ref{f:Deltaprime}).

\begin{figure}[hbt]
\vspace{-2.5em}
\begin{center}
\psfrag{ep}{$\ep$}
\psfrag{2ep}{$2\ep$}
\psfrag{O}{}
\psfrag{1}{}
\psfrag{slope}{\scriptsize slope = $N/\be = \ga$}
\psfrag{Trig}{$\Delta'$}
\psfrag{ell}{$\ell$}

\psfrag{slope2}{\scriptsize slope = $N/\be = \ga$}
\psfrag{slope3}{\scriptsize slope = $\ga$}
\psfrag{FC}{$F_{C}$}
\psfrag{FD}{$F_{D}$}
\psfrag{height2}{\scriptsize $r-\frac{1}{2}$}
\psfrag{height3}{\scriptsize $r'$}
\psfrag{Del'}{$\Delta'$}
\psfrag{ell}{$\ell$}

\epsfig{file=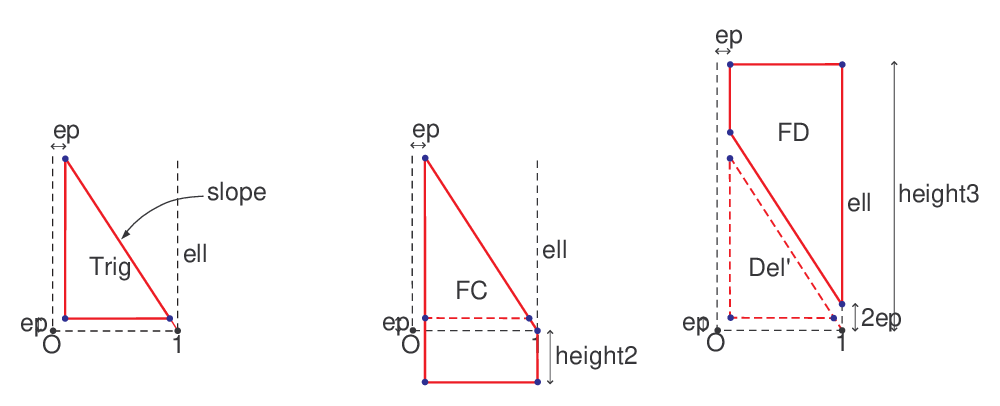, height=6cm}
\end{center}
\vspace{-1em}
\caption{The triangle $\Delta'$ and trapezoids $F_{C},F_{D}$.}
\label{f:Deltaprime}
\vspace{-1em}
\end{figure}

\bigskip

Let us show that these trapezoids encode the function as stated above.
For $F_{C}$, we have $(F_{C} + t \vec v) \cap \ell = \big[\frac{1}{2}-r,\, \frac{t}{\be}\big]$, and thus $|F_{C} + t \vec v| = r + \floor{t/\be}$.
Similarly, for $F_{D}$, the hypotenuse of $(F_{D} + t \vec v)$ intersects $\ell$ at $y = t/\be + 2\ep$, and thus $(F_{D} + t \vec v) \cap \ell = \big[\frac{t}{\be} + 2\ep,\, r']$.
Since $t/\be < t/\be + 2\ep < (t+1)/\be$, we have $|F_{D} + t \vec v| = r' - \floor{t/\be}$, as desired.

\medskip

Note that the counting function for each constructed trapezoid is periodic modulo~$N$.
In other words, \ts $|F_{A} + t \vec v| = |F_{A} + (t \text{ mod } N) \ts \vec v|$ for every $t \in \Z$,
and the same result holds for $F_{B},F_{C},F_{D}$.
From this point on, we let $t$ take values over $\Z$ in place of our
earlier restriction \ts $t\in [0,N)$.

\subsection{The product construction}\label{sec:prod}
The next step is to construct polytopes that encode products functions of the form \ts
$p \pm q t$ \ts and \ts $r \pm \floor{t/\be}$.

Consider any $d$ functions $h_{1}(t),\dots,h_{d}(t)$ of these forms.
We take the trapezoids $F_{1}, \dots, F_{d}$ whose counting functions encode $h_{i}$'s.
Each $F_{i} \subset \R^{2}$ is described by a system:
$$F_{i} \; = \; \big\{ (x,y) \in \R^{2} \; : \; \mu_{i} \le x \le \nu_{i},\; \rho_{i} + \tau_{i} x \le y \le \rho_{i}' + \tau_{i}'x \big\}.$$
We embed $F_{i}$ into the 2-dimensional subspace spanned by coordinates $x,y_{i}$ inside $\R^{d+1}$ (with coordinates $x,y_{1},\dots,y_{d}$).
Then define:
\begin{equation}\label{eq:product}
P \;=\; \big\{\ts(x,y_{1},\dots,y_{d}) \in \R^{d+1} \;:\; \max_{1 \le i \le d}\mu_{i} \le x \le \min_{1 \le i \le d} \nu_{i},\;
 \rho_{i} + \tau_{i} x \le y_{i} \le \rho_{i}' + \tau_{i}'x
\ts \big\}.
\end{equation}
It is clear that for every $t$ and every vertical hyperplane $H = \{x=x_{0}\}$ in $\R^{d+1}$, we have $(P + t \vec v) \cap H = \big((F_{1} + t \vec v\big) \cap H) \times \dots \times \big((F_{d} + t \vec v) \cap H\big)$.\footnote{Note that each $F_{i} + t \vec v$ intersects exactly one such hyperplane $H$ with $x_{0} \in \Z$.}
Therefore, we have $$|P \cap t \vec v| \;=\; |F_{1} \cap t \vec v| \dots |F_{d} \cap t \vec v| \;=\; h_{1}(t) \dots h_{d}(t).$$
So the $(d+1)$-dimensional polytope $P$ encodes the product $h_{1}(t) \dots h_{d}(t)$.
Note that $P$ is combinatorially a cube, which means it has $2(d+1)$ facets and $2^{d+1}$ vertices.

\subsection{Putting it all together}\label{s:embed_argument}
We apply this product construction to each of the four terms $T_{1},T_{2}$ in~\eqref{eq:T12},
$T_{3},T_{4}$ in~\eqref{eq:T34} and get four polytopes
$P_{1} \in \R^{5}$, $P_{2} \in \R^{3}$, $P_{3} \in \R^{4}$, $P_{4} \in \R^{2}$ such that
\begin{equation}\label{eq:P_iT_i}
|P_{i} + t \vec v| \. = \. T_{i}(t \mmod N) \quad \text{for every $t \in \Z$.}
\end{equation}
Now we embed them into $\R^{6}$ as follows:
\begin{equation}\label{eq:Q_i}
\aligned
Q_{1} \; &= \; \{\x \in \R^{6} \; : \; (x_{1},\dots,x_{5}) \in P_{1} ,\;  x_{6}=1 \},\\
Q_{3} \; &= \; \{\x \in \R^{6} \; : \; (x_{1},\dots,x_{4}) \in P_{3} ,\;  x_{5}=1,\; x_{6}=0 \},\\
Q_{2} \; &= \; \{\x \in \R^{6} \; : \; (x_{1},\dots,x_{3}) \in P_{2} ,\;  x_{4}=1,\; x_{5}=0,\; x_{6}=0 \},\\
Q_{4} \; &= \; \{\x \in \R^{6} \; : \; (x_{1},x_{2}) \in P_{4},\;  x_{3}=1,\; x_{4}=0,\; x_{5}=0,\; x_{6}=0 \}.
\endaligned
\end{equation}
Define the polytope
\begin{equation}\label{eq:W}
W = \conv(Q_{1},\dots,Q_{4}).
\end{equation}
First, note that $Q_{1},\dots,Q_{4}$ are disjoint. They also have the property that for every $t \in \Z$:
$$
(W + t \vec v) \cap \Z^{6} \; = \; \bigsqcup_{i=1}^{4} \Big(\ts (Q_{i} + t \vec v) \cap \Z^{6} \ts\Big).
$$
To see this, consider some lattice point $\z = (z_{1},\dots,z_{6}) \in (W + t \vec v) \cap \Z^{6}$.
Since $W$ is the convex hull of  $Q_{1},\dots,Q_{4}$, and each $Q_{i}$ sits in one of the two hyperplanes $\{x_{6}=1\}$, $\{x_{6}=0\}$, we must have $z_{6}=1$ or $z_{6}=0$.
This means $\z \in (Q_{1}+t\vec v)$ or $\z \in (\conv(Q_{2},Q_{3},Q_{4}) + t\vec v)$.
Assume the latter case, then we continue considering the coordinate $z_{5}$.
By a similar argument, we must have $\z \in (Q_{3} + t \vec v)$ or $\z \in (\conv(Q_{2},Q_{4}) + t\vec v)$.
For the latter case, the coordinate $z_{4}$ should finally tell us either $\z \in (Q_{2} + t \vec v)$ or $\z \in (Q_{4} + t \vec v)$.

Thus, for every $t \in \Z$, we have:
$$|W + t \vec v| \; = \; \sum_{i=1}^{4} \. |Q_{i} + t \vec v| \;=\;
 \sum_{i=1}^{4} \. |P_{i} + t \vec v| \;=\; \sum_{i=1}^{4} \. T_{i}(t \mmod N) \; = \; g(t \mmod N).$$
By ~\eqref{eq:gT}, we conclude that computing the following minimum is $\NP$-hard:
$$\min_{t \in \Z} \. |W + t \vec v| \; = \; \min_{0 \le t < N} \ts g(t).
$$
Note that the polytopes $Q_{1},Q_{2},Q_{3},Q_{4}$ have $32,8,16,4$ vertices, respectively.
Thus, the polytope~$W$ has in total $60$ vertices, and satisfies Theorem~\ref{th:int_min}. \hfill $\sq$

\bigskip

\section{Proof of Theorem~\ref{th:real_min}} \label{s:real}
We modify the construction in the proof of Theorem~\ref{th:int_min} by perturbing all its
ingredients to ensure that the desired minimum coincides with the one in the integer case.
This construction is rather technical and assumes the reader is familiar with
details in the proof above.

\smallskip

Recall that $0 \le \al,\ga < \be$, $N = \be\ga$, $\ep = 1/4N^{2}$ and $\vec v = \vec e_{1}/N$.
We perturb all constructed trapezoids as follows.
Denote by $s$ the maximum slope over all hypotenuses of all constructed trapezoids.
By a quick inspection of the terms $T_{1},\dots,T_{4}$ in~\eqref{eq:T12} and~\eqref{eq:T34},
one can see that \ts $s < 4 \be^{4} N < 4\be^{6}$.
Take $\de > 0$ much smaller than $\ep$ and $(\be\ts s)^{-1}$.
For example, $\de := 1/4\be^{8}$ works.
Now translate each constructed trapezoid $F$ by a distance $+\delta$ horizontally in $\R^{2}$.
Let~$F'$ be such a translated copy of some $F$.\footnote{Recall that each $F$ encodes some function $h(t)$ as $|F + t \vec v| = h(t \mmod N)$ for every $t \in \Z$.}
Then it is not hard to see that $|F + t \vec v| = |F' + t \vec v|$ for all $t \in \Z$.
In fact, due to the $\de$ perturbation, we have:
$$
|F + t \vec v| \; = \; |F' + t \vec v| \; = \; \Big|F' + \Big(\frac{t}{N} + \tau\Big) \vec e_{1}\Big|
$$
for every $t \in \Z$ and $\tau \in [-\de/4, \de/4]$.
This can be checked directly for all the trapezoid of types $F_{A},F_{B},F_{C}, F_{D}$
constructed in the proof of Theorem~\ref{th:int_min}.
Define the real set
\begin{equation}\label{eq:Z_de}
Z_{\de} \; = \; \Big\{\frac{t}{N} + \tau \; : \; t \in \Z,\, -\de/4 \le \tau \le \de/4 \Big\}.
\end{equation}
For $\lambda \in Z_{\de}$, denote by $t(\lambda)$ the (unique) integer $t$ such that $|\lambda - t/N| \le \de/4$.
By the above observations, we have $|F' + \lambda \vec e_{1}| = |F + t(\lambda) \vec v|$ for every $\lambda \in Z_{\de}$.
Now we take these perturbed trapezoids and construct $P_{1}', \dots, P_{4}'$ as similar to $P_{1},\dots,P_{4}$ above, using the same product construction (see~\eqref{eq:product}).
Note that $P'_{i} = P_{i} + \de \vec e_{1}$ and by~\eqref{eq:P_iT_i}, for every $\lambda \in Z_{\de}$ we have:
\begin{equation}\label{eq:P'_iP_i}
|P'_{i} + \lambda \vec e_{1}| \; = \; |P_{i} + t(\lambda) \vec v| \; = \; T_{i}(t(\lambda) \mmod N) \quad (1 \le i \le 4).
\end{equation}

\medskip

We need to ``patch up'' $Z_{\de}$ to make it the whole real line $\R$.
Let

\begin{equation}\label{eq:Y_de}
Y_{\de} \; = \; \Big\{\frac{t}{N} + \tau \; \, : \, \; t \in \Z, \ \frac{\de}{8} \le \tau \le \frac{1}{N}-\frac{\de}{8} \Big\}.
\end{equation}

\begin{figure}[hbt]
  \vspace{-1em}
\begin{center}
\psfrag{Z}{$Z_{\de}$}
\psfrag{Zprime}{$Y_{\de}$}
\psfrag{0}{$0$}
\psfrag{1}{$\frac{1}{N}$}
\psfrag{2}{$\frac{2}{N}$}
\psfrag{N-1}{$\frac{N-1}{N}$}
\psfrag{N}{$1$}
\psfrag{-1}{-$\frac{1}{N}$}
\psfrag{...}{$\dots$}
\psfrag{l2}{$\frac{\de}{4}$}

\epsfig{file=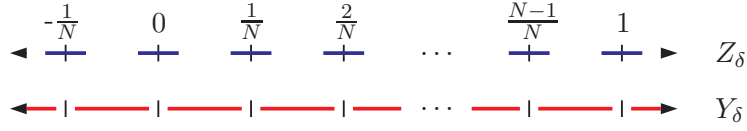, height=2cm}
\end{center}
\vspace{-1em}
\caption{The sets $Z_{\delta}$ and $Y_{\delta}$ consisting of bold segments.}
\label{f:ZZprime}
\end{figure}

\nin
It is clear that $Z_{\de} \cup Y_{\de} = \R$.
Take a large constant $\KK$, s.t. $\KK \gg g(t)$ for all $0 \le t < N$.
For example, $\KK := 10\be^{10}$ \ts will suffice for our purposes,
by~\eqref{eq:T12}--\eqref{eq:gT}.
Now consider the following parallelogram:
$$
R \; = \; \Big\{ (x,y) \in \R^{2} \; : \;  \KK{}N-\sfrac{1}{2} \ge y \ge 0,\; 1 - \frac{\de}{8} -\frac{y}{N} \ge x \ge 1 - \frac{1}{N} + \frac{\de}{8}  - \frac{y}{N}  \Big\}
$$
(see Figure~\ref{f:R}).

\begin{figure}[hbt]
\vspace{-1em}
\begin{center}
\psfrag{1}{$1$}
\psfrag{0}{$0$}
\psfrag{d1}{\scriptsize$\de/8$}
\psfrag{d2}{\scriptsize$1/N-\de/8$}
\psfrag{R}{$R$}
\psfrag{slope}{$\text{slope}\; \frac{y}{x} = N$}
\psfrag{h}{$\KK{}N - \frac{1}{2}$}

\epsfig{file=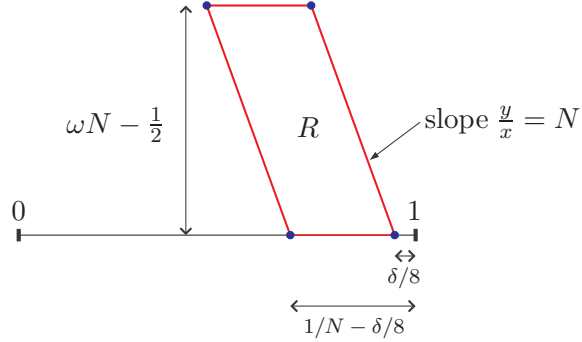, height=5cm}
\end{center}
\vspace{-1em}
\caption{The parallelogram $R$.}
\label{f:R}
\vspace{-.5em}
\end{figure}

\begin{lemma}\label{l:real}
We have: \ts $|R + \lambda \vec e_{1}| = \KK$ if $\lambda \in Y_{\de}$, and $|R + \lambda \vec e_{1}| = 0$ otherwise.
\end{lemma}

\begin{proof}
Denote by $R_{(i)}$ the horizontal slice of $R$ at height $i \in \Z$.
Then for the bottom edge $R_{(0)}$,  we have  $|R_{(0)} + \lambda \vec e_{1}| = 1$ if $\de/8 \le \lambda \mmod 1 \le 1/N - \de/8$, and $|R_{(0)}+\lambda \vec e_{1}| = 0$ otherwise.
In other words, $|R_{(0)} + \lambda \vec e_{1}| = 1$ if and only if $\lambda$ lies in some $jN$-th segment of $Y_{\de}$ ($j \in \Z$).
Also every next slice is translated by $-1/N$, i.e., $R_{(i+1)} = R_{(i)} - \vec e_{1} / N$.
There are in total $\KK{}N$ non-empty slices, which implies the claim.
\end{proof}

\medskip

Recall the perturbed polytopes $P'_{1}, \dots, P'_{4}$ above, see~\eqref{eq:P'_iP_i}.
We embed them into $\R^{5}$ similarly to~\eqref{eq:Q_i}:
\begin{equation}\label{eq:Q'_i}
\aligned
Q'_{1} \; &= \; \{\x \in \R^{6} \; : \; (x_{1},\dots,x_{5}) \in P'_{1} ,\; x_{6}=1 \},\\
Q'_{3} \; &= \; \{\x \in \R^{6} \; : \; (x_{1},\dots,x_{4}) \in P'_{3} ,\; x_{5}=1,\; x_{6}=0 \},\\
Q'_{2} \; &= \; \{\x \in \R^{6} \; : \; (x_{1},\dots,x_{3}) \in P'_{2} ,\; x_{4}=1,\; x_{5}=0,\; x_{6}=0 \},\\
Q'_{4} \; &= \; \{\x \in \R^{6} \; : \; (x_{1},x_{2}) \in P'_{4},\; x_{3}=1,\; x_{4}=0,\; x_{5}=0,\; x_{6}=0 \}.\\
\endaligned
\end{equation}
We also embed $R$ into $\R^{5}$ as:
$$
Q'_{5} \; = \; \{\x \in \R^{6} \; : \; (x_{1},x_{2}) \in R,\; \quad  x_{3}=0,\; x_{4}=0,\; x_{5}=0,\; x_{6}=0 \}.
$$
Let $W' = \conv(Q'_{1}, \dots, Q'_{5})$. By the above embeddings, we have:
$$
|W' + \lambda \vec e_{1}| \; = \; \sum_{i=1}^{5} \. |Q'_{i} + \lambda \vec e_{1}|
\; = \; |R + \lambda \vec e_{1}| \;+\; \sum_{i=1}^{4} \. |P'_{i} + \lambda \vec e_{1}|.
$$
See Section~\ref{s:embed_argument} for an explanation on the additivity of the counting functions. Now if $\lambda \in Y_{\de}$, we have:
$$
|W' + \lambda \vec e_{1}| \; \ge \;  |R + \lambda \vec e_{1}| \; = \; \KK \; \gg \; \max_{0 \le t < N} \ts g(t).
$$
On the other hand, if $\lambda \notin Y_{\de}$, then $\lambda \in Z_{\de}$ \ts by~\eqref{eq:Z_de} and~\eqref{eq:Y_de}.
In this case, by~\eqref{eq:P'_iP_i} and Lemma~\ref{l:real}, we have:
$$
|W' + \lambda \vec e_{1}| \; = \; \sum_{i=1}^{4} \.
|P'_{i} + \lambda \vec e_{1}| \; = \; \sum_{i=1}^{4} \. T_{i}(t(\lambda) \mmod N)
\; = \; g(t(\lambda) \mmod N).
$$
We conclude that the following minimum is $\NP$-hard to compute:
$$
\min_{\lambda \in \R} \. |W' + \lambda \vec e_{1}| \; = \; \min_{0 \le t < N} \ts g(t).
$$
Note that the polytopes $Q'_{1},Q'_{2},Q'_{3},Q'_{4},Q'_{5}$ have
$32,8,16,4,4$ vertices, respectively.
Thus, polytope $W'$ has in total $64$ vertices.  This completes
the proof  of Theorem~\ref{th:real_min}.
\hfill $\sq$

\bigskip

\section{Applications}\label{s:app}

\subsection{Proof of Theorem~\ref{t:Ehrhart}}\label{s:Ehrhart_proof}
Recall from  Section~\ref{sec:int_min} the polytope $W \subset \R^{6}$ with 60 vertices and the translation vector $\vec v = \vec e_{1}/N$.
From now on, we refer to $W$ as $P$ (its intended role in Theorem~\ref{th:int_min}).
From the construction in Section~\ref{sec:int_min}, $P$ is a closed polytope containing at least one integer point, which we call $\vec p \in \Z^{n}$.
  We translate $P$ by $-\vec p$ so that $\vec 0 \in P$, meanwhile still keeping $|P + t\vec v|$ the same for every $t \in \Z$.

  Consider a very large multiple $M$ of $N$ (quantified later).
  For some $0 \le t < N$, consider the two polytopes
$$R_{t} = P + (t+M)\vec v \quad \text{and} \quad R'_{t} = \frac{t+M}{M}P + (t+M) \vec v$$
  First note that these are closed polytopes with $R_{t} \subset R'_{t}$.
  Also since $N | M$ and $N\vec v = (1,0,\dots)$, $R_{t}$ is just an integer translate of $P + t \vec v$.
  Thus, the distance from $R_{t}$ to its closest outer integer point is exactly the same as that for $P + t \vec v$.
So if $R'_{t}$ is only slightly larger than $R_{t}$, they should contain the same set of integer points, which is again an integer translate of the set $(P + t\vec v) \cap \Z^{n}$.
Therefore, if $M \gg N > t$, we should have $|R_{t}| = |R'_{t}|$.

To ensure  $|R_{t}| = |R'_{t}|$ for all  $0 \le t < N$, it suffices to have $N/M < d_{1}/D_{2}$, where:
$$d_{1} = \min_{0 \le t < N}\delta\big(P+t\vec v,\; \Z^{6} \cpl (P + t\vec v)\big) \quad \text{and} \quad D_{2} = \text{diameter of }P.$$
Here $\delta(\cdot,\cdot)$ denotes the shortest distance between $2$ sets.
Both $1/d_{1}$ and $D_{2}$ are polynomially bounded in $N$ and the largest $pq$ over all vertex coordinates $p/q$ of $P$ (see~\cite[Ch.10]{Schrijver}).
So $M$ only needs to be polynomially large in $N$ and the coordinates of $P$.

Now we have $|R_{t}| = |R'_{t}|$ for every $0 \le t < N$.
Let $Q = \frac{1}{M}P + \vec v$, then $R'_{t} = (t+M)Q$.
Thus, $|R_{t}| = |(t+M)Q|$ for every $0 \le t < N$.
Recall that $|P + t \vec v|$ is periodic modulo $N$ and $N | M$.
So $|R_{t}| = |P + (t+M)\vec v| = |P + t \vec v|$ for very $t$.
We conclude that
$$|P+t\vec v| = |(t+M)Q| \quad\text{for every}\quad 0 \le t < N.$$
Thus, computing $\min_{0 \le t < N} |(t+M)Q| = \min_{0 \le t < N} |P + t \vec v|$ is $\NP$-hard.

By binary search, finding $\min_{0 \le t < N}|(t+M)Q|$ is equivalent to deciding polynomially many sentences of the from $\min_{0 \le t < N}|(t+M)Q| < k$ for varying $k$.
From the definition of \textsc{$k$-ETP}, we have $\min_{0 \le t < N}|(t+M)Q| < k$ if and only if $g(Q,k) \ge M$.
This implies that computing $g(Q,k)$ is $\NP$-hard. \hfill $\sq$

\medskip

\subsection{Proof of Theorem~\ref{th:fluctuation}}
The constants $K,M$ will be later quantified.
Recall that
\begin{equation}\label{eq:pt}
\. p(t) \. = \. \sum_{i=1}^{r} \. \gamma_{i} \prod_{j=1}^{n} \bigl\lfloor\al_{ij}t + \be_{ij}\bigr\rfloor\ts
\end{equation}
with $\gamma_{i} \in \Z$.
By increasing $n$ by $1$ and writing $\gamma_{i} = \floor{0 t + \gamma_{i}}$, we can assume that all coefficients $\gamma_{i}=1$.
Let $\vec v = \vec e_{1}/N$.
First, we construct a polytope $W \subset \R^{d}$ such that $p(t \mmod N) + K = |W + t \vec v|$ for all $t \in \Z$.
We need a technical lemma:

\begin{lem}\label{l:alg_id}
For every $n \ge 2$,  we have the algebraic identity:
\begin{equation}\label{eq:alg_id}
  3^{n-1} g_{1} \cdots g_{n} + h_{1} \cdots h_{n} = \sum_{\substack{S \subseteq [n],\, S \neq \varnothing}} 3^{\delta(S)} \prod_{i \in [n] \cpl S}g_{i} \, \cdot \, \prod_{j \in S}(g_{j} \; + \sigma_{j}(S) \. \tau_{j}(S) \. h_{j})
\end{equation}
  where
  $$
  \sigma_{j}(S) = \Big\{
  \begin{matrix}
1 &\text{if} & j-1 \in S\\
-1 &\text{if} & j-1 \notin S
\end{matrix}
\; , \;\;
  \tau_{j}(S) = \Big\{
  \begin{matrix}
1 &\text{if} & \max(S) > j \\
-1 &\text{if} & \max(S) = j
\end{matrix}
  \; , \;\;
  \delta(S) = \max(0,n-\max(S)-1).
   $$
\end{lem}

\begin{proof}
We show the indentity by induction. The base case $n=2$ can be easily checked:
\begin{equation}\label{eq:base_case}
3g_{1}g_{2} + h_{1}h_{2} = (g_{1}+h_{1})(g_{2}+h_{2}) + g_{1}(g_{2}-h_{2}) + g_{2}(g_{1}-h_{1}).
\end{equation}
Assume~\eqref{eq:alg_id} holds up to $n-1$, we show it for $n$.
First, by substituting $3^{n-2}g_{2} \dots g_{n}$ for $g_{2}$ and $h_{2} \dots h_{n}$ for $h_{2}$ in~\eqref{eq:base_case}, we have:
\begin{equation*}
\aligned
3^{n-1}g_{1} \dots g_{n} + h_{1} \dots h_{n} \; = \;\; &(g_{1}+h_{1})\underbrace{(3^{n-2}g_{2} \dots g_{n} + h_{2} \dots h_{n})}_{A} \; + \\
&g_{1}\underbrace{(3^{n-2}g_{2}\dots g_{n} - h_{2} \dots h_{n})}_{B} \, + \, \underbrace{3^{n-2} g_{2} \dots g_{n}(g_{1} - h_{1})}_{C}.\\
\endaligned
\end{equation*}
Now for term $A$, we directly apply~\eqref{eq:alg_id} for $n-1$ and variables $g_{2},\dots,g_{n}$, $h_{2},\dots,h_{n}$.
For term $B$, we change variable from $-h_{2}$ to $h_{2}$ and also apply~\eqref{eq:alg_id} with $n-1$.
Term $C$ corresponds to that in~\eqref{eq:alg_id} with $S = \{1\}$.
The sign functions $\sigma_{j}(S)$ and $\tau_{j}(S)$ can be understood as follows.
If $j = \max(S)$ then $h_{j}$ switches sign, just like $h_{1}$ in term $C$.
If $j-1 \in S$ then $h_{j}$ does not switch sign, just like $h_{2}$ in term $A$.
If $j-1 \notin S$ then $h_{j}$ does switch sign, just like $h_{2}$ in term $B$.
Finally, the function $\delta(S)$ corresponds to the power $3^{n-2}$ in term $C$.
\end{proof}

The point of Lemma~\ref{l:alg_id} is that if $q_{i}(t) = \prod_{j=1}^{n}h_{ij}(t)$, where $h_{ij}(t) = \floor{\al_{ij}t + \be_{ij}}$, and $g \in \N$ is big enough then we can write:
\begin{equation}\label{eq:sum}
q_{i}(t) + 3^{n-1} g^{n} = h_{i1}(t) \dots h_{in}(t) + 3^{n-1} g^{n} = \sum_{S \subseteq [n],\, S \neq \varnothing} 3^{\delta(S)} \ts g^{n - |S|} \ts \prod_{j \in S} \big(g \,\pm\, h_{ij}(t)\big).
\end{equation}
Now the trapezoid construction from Section~\ref{sec:int_min} can be applied to each term $g \pm h_{ij}(t)$.
In other words, for each $j$, we construct two trapezoids $F^{+}_{ij}$ and $F^{-}_{ij}$ so that:
$$|F^{+}_{ij} + t \vec v| = g + h_{ij}(t \mmod N) \quad \text{and} \quad |F^{-}_{ij} + t \vec v| = g - h_{ij}(t \mmod N) \quad \text{for every} \quad t \in \Z.$$
For each $S \subseteq [n]$ in the sum in~\eqref{eq:sum}, we take product of the trapezoids for the terms $g \pm h_{ij}(t)$ with the construction from Section~\ref{sec:prod}.
This results in some polytope $P'_{S}$ in $\R^{|S|+1}$ with $2^{|S|+1}$ vertices.
Then we take a prism of height $3^{\delta(S)} g^{n-|S|}$ over $P'_{S}$ to get a polytope $P_{S} \in \R^{|S|+2}$ with $2^{|S|+2}$ vertices such that:
$$
|P_{S} + t \vec v| \; = \; 3^{\delta(S)} g^{n-|S|} \prod_{j \in S} \big( g \,\pm\, h_{ij}(t \mmod N) \big) \quad \text{for every} \quad t \in \Z.
$$
By padding in extra dimensions, we can assume each $P_{S} \subset \R^{n+2}$.
To sum over all $S$ ($2^{n}-1$ of them), we pad in another extra $n$ dimensions, and augment each $P_{S}$ with the coordinates of a distinct point in $\{0,1\}^{n}$ similarly to~\eqref{eq:Q_i}.
The resulting polytopes $Q_{S} \subset \R^{2n+2}$ still satisfy $|P_{S}|=|Q_{S}|$.
We then take the convex hull of all $Q_{S}$ to get a polytope $W_{i} \subset \R^{2n+2}$ such that:
$$
|W_{i} + t \vec v| \; = \; \sum_{S \subseteq [n],\, S \neq \varnothing} 3^{\delta(S)} \ts g^{n - |S|} \ts \prod_{j \in S} \big( g \,\pm\, h_{ij}(t \mmod N) \big) \; = \; q_{i}(t \mmod N) + 3^{n-1}g^{n}.
$$
for ever $t \in \Z$.
See Section~\ref{s:embed_argument} for an explanation of the additivity in the counting functions.
Note that $W_{i}$ has at most $(2^{n}-1) 2^{n+2} < 4^{n+1}$ vertices.

Now we have a polytope $W_{i} \subset \R^{2n+2}$ for each term $q_{i}(t) = \prod_{j=1}^{n} \floor{\al_{ij}t + \be_{ij}}$ in~\eqref{eq:pt}.
Again, to sum up $q_{i}$ over $1 \le i \le r$, we pad each $W_{i}$ with $\lceil\log r\rceil$ extra dimensions and augment it with a distinct point in $\{0,1\}^{\lceil \log r \rceil}$.
Taking their convex hull, we get $P \subset \R^{d}$ such that
$$p(t \mmod N) + r 3^{n-1} g \; = \; |P + t \vec v| \quad \text{for every} \quad t \in \Z.$$
Here $d = 2n + 2 + \lceil \log r \rceil$ is the dimension, and $P$ has at most $r 4^{n+1}$ vertices.
In this construction, we only need $g > |h_{ij}(t)|$ for all $1 \le i \le r, 1 \le j \le n$ and $0 \le j < N$. So $g = 2 \lceil \max |\al_{ij}| N + \max |\be_{ij}| \rceil$ suffices.
We let $K = r3^{n-1}g$.

Finally, the argument from Section~\ref{s:Ehrhart_proof} can be applied to $P$.
 This gives a polytope $Q \subset \R^{d}$ (with the same number of vertices) and an $M \in \N$ so that:
$$
p(t) + K  \;=\;  |P + t \vec v| \;=\; |(t+M)Q| \;=\; f_{Q}(t+M) \quad \text{for every} \quad 0 \le t < N.
$$
This finishes the proof of Theorem~\ref{th:fluctuation}. \hfill $\sq$

\bigskip

\section{Integer polytopes} \label{s:int-pol}
While much of the paper deals with rational polytopes in fixed dimensions,
we can ask similar questions about \emph{integer polytopes}
(polytopes with vertices in $\zz^n$).

\begin{prop}
For integer polytopes, the \textsc{$k$-ETP} problem can be solved in polynomial
time.
\end{prop}

\begin{proof}
The Ehrhart polynomial $f_{P}(t)$ of an integer polytope $P \subset \R^{n}$
is a monotone polynomial of degree at most~$n$, see e.g.~\cite{B3}.  
Since $n$ is fixed, the coefficients of $f_{P}(t)$ can be computed 
using Lagrange interpolation.  Now apply the binary search to solve 
the \textsc{$k$-ETP} problem from definition. 
\end{proof}

Note that this approach also extends to (rational) polytopes $P$ with a fixed
\emph{denominator}, defined as the smallest $t \in \Z_{+}$ such that $tP$ is integer.

\smallskip

For \textsc{Polytope Translation}, we do not know if 
Theorem~\ref{th:real_min} continues to hold for integer polytopes.
However, it is not difficult to see that Theorem~\ref{th:int_min} extends 
to this setting:

\begin{theo}\label{th:int_min-2}
Given an integer polytope $P \subset \R^{6}$ with at most $64$ vertices
and an integer $N \in \N$,
minimizing \ts $|P + t \vec e_{1}/N|$ \ts over $t \in \Z$ is $\NP$-hard.
\end{theo}

\begin{proof}[Sketch of proof]
The trapezoids in Section~\ref{sec:trapezoid} can be reused, with the $\ep$'s
removed to make all their vertices integer.\footnote{Those $\ep$'s only mattered
in Section~\ref{s:real}, where we say that small perturbation does not change
the number of integer points in the trapezoids.}  A small complication arises
for trapezoids of type $F_{D}$ in Figure~\ref{f:Deltaprime}, because now
\ts $|F_{D} + t \vec v| = r' - \floor{(t-1)/\be}$ \ts instead of \ts
$r' - \floor{t/\be}$. This is easily circumvented by considering only $t \in [0,N)$ s.t.\
$\be \nmid t$, and thus \ts $\floor{(t-1)/\be} = \floor{t/\be}$.
The remaining $t \in [0,N)$ with $\be | t$ can be ignored because they correspond
to $v = 0$ in~\eqref{eq:QDE_restated}, which can be checked directly.
\end{proof}

\smallskip

For the special case of \emph{integer polygons}, the number of integer points
vary quite nicely under translation (cf.~\cite{EH}).

\begin{prop}
For every fixed $m$, the \textsc{Polytope Translation} problem for integer $m$-gons
can be solved in polynomial time.
\end{prop}

\begin{proof}
Let $Q\ssu \rr^2$ be an integer $m$-gon.  Then \ts $f(\la):=|Q + \lambda \vec e_{1}|$
\ts is a sum of at most $m$ terms of the form \ts
$\bigl(a_{i} + b_{i} \floor{c_{i}\. \lambda}\bigr)$, for some $a_i,b_i,c_i\in\qqq$.  
Then the generating function
$$
F_{Q,N}(z,w) \. := \. \sum_{k=0}^{N-1} \. z^k \. w^{f(k/N)}
$$
can be written in the \emph{short GF form} (see~\cite{BP,BW}).
Here $1/N$ is a small enough refinement of the unit interval.
Then the short GF technique of taking projections can be applied
to $F_{Q,N}(z,w)$ to find the minimum of $f(k/N)$ in polynomial time.
We omit the details. \end{proof}

Curiously, Alhajjar proved in~\cite[Prop.~4.15]{Alh},
that for every integer polygon $Q\ssu \rr^2$, the corresponding maximization problem is trivial:
$$
|Q| \. > \. |Q+ \lambda \vec e_1|\ts,  \quad\text{for all} \ \, 0< \la< 1.
$$
This does not extend to $\rr^3$, however.  For example, take $\De\ssu \rr^3$
defined as the convex hull of points $(0,0,0)$, $(1,0,0)$, $(0,1,k)$ and $(1,-1,k)$.
Then $|\De|=4$, while $\bigl|\De+(1/2,0,0)\bigr|=k+1$, which is unbounded.

Finally, let us mention a large body of work on coefficients of the
\emph{$h^\ast$-vector} for the Ehrhart polynomials of integer polytopes.
This gives further restrictions on the values $f_Q(t)$ as in
Corollary~\ref{c:univ}.  We refer to~\cite{Braun} for a
recent survey article and references therein.

\bigskip

\section{Final remarks and open problems}\label{s:fin-rem}

\subsection{}\label{ss:finrem-min}
Now that  \textsc{Polytope Translation} is \NP-hard,
it would be interesting to know its true complexity.  First, it is clearly
in \PSPACE.  Also our proof is robust enough to allow embedding
of general polynomial optimization decision problems (cf.~\cite{DHKW}).
Although we were unable to find a more general optimization problem that fits our framework, we hope to return to this in the future.

Note that in computational complexity, counting oracles are extremely powerful,
as shown by Toda's theorem (see e.g.~\cite{AB,Pap}).  From this point
of view, our Theorem~\ref{th:real_min} is unsurprising, since it uses
a counting oracle in a restricted setting.

\medskip

\subsection{}
In another direction, it would be interesting to see if \textsc{Polytope Translation} remains \NP-hard in lower dimensions.
We believe that dimension~$6$ is Theorem~\ref{th:real_min}
is not sharp.

\begin{conj}  \label{conj:PT-simplex}
The \textsc{Polytope Translation}
problem for rational polytopes $P\ssu \R^3$ is \NP-hard.
\end{conj}

In the plane, the polygon translation problem (with a fixed number
of vertices) seem to have additional structures that prevent it from
being computationally hard.
In the special case of rational trapezoids,
it can be reduced to a Diophantine approximation problem of
unknown complexity (see the approach in~\cite{EH}).  We conjecture that
the polygon translation problem is intermediate between $\poly$ and $\NP$.

\medskip

Similarly, we believe that hardness still holds for much simpler types of polytopes:

\begin{conj}  For some fixed $n$, the \textsc{Polytope Translation}
problem for rational simplices $\De\ssu \R^n$ is \NP-hard.
\end{conj}

By analogy,
we believe that Theorem~\ref{t:Ehrhart} also holds for simplices:

\begin{conj}\label{conj:ETP-simplex}
\textsc{$k$-ETP} is $\NP$-hard for rational simplices $\De \in \R^{n}$,
for some fixed~$n$.
\end{conj}

A significantly stronger result would be the following:

\begin{conj}  The \textsc{$k$-Frobenius Problem}
 is \NP-hard for some fixed~$n$.
\end{conj}


\medskip

\subsection{}
Corollary~\ref{c:univ} is the type of universality result
which occasionally arise in discrete and algebraic geometry 
(see e.g.~$\S\S$12,13 in~\cite{Pak} and references therein).  
It would be interesting to find a simple or more direct 
proof of this result.  In fact, we conjecture that the dimension 
bound $d=O(\log r)$ is sharp, cf.\ Prop.~8.1 in~\cite{KannanNPC}. 

\vskip.72cm


{\small

\subsection*{Acknowledgements}
We are thankful to Fritz Eisenbrand for telling us about the 
\textsc{IPM} problem, to Sasha Barvinok for introducing us to the subject,
and to Jes\'{u}s De Loera for his insights and encouragement.
We are also grateful to Elie Alhajjar, Matt Beck, Lenny Fukshansky,
Robert Hildebrand, Ravi Kannan, Oleg Karpenkov, Matthias K\"oppe
and Kevin Woods for interesting conversations and helpful remarks.
This work was initiated while both authors were in residence
of the MSRI long term Combinatorics program in the Fall of 2017;
we thank MSRI for the hospitality.  The first author
was partially supported by the UCLA Dissertation Year Fellowship.
The second author was partially supported by the~NSF.

}

 \end{document}